\documentclass[11pt,reqno]{amsart}
\usepackage[T1]{fontenc}
\usepackage{lmodern}
\usepackage[margin=1.05in]{geometry}
\usepackage{amssymb,mathtools,microtype,needspace}
\usepackage[hidelinks]{hyperref}
\newtheorem{theorem}{Theorem}[section]
\newtheorem{lemma}[theorem]{Lemma}

\newtheorem{corollary}[theorem]{Corollary}
\newtheorem{conjecture}[theorem]{Conjecture}
\DeclareMathOperator{\Var}{Var}
\DeclareMathOperator{\Cov}{Cov}
\newcommand{\E}{\mathbb{E}}
\newcommand{\sqcupdot}{\mathbin{\dot\cup}}
\title{Optimally large cuts in both directions}
\author{Frederic Folz}
\address{Altifina Discovery Lab, 66802 Überherrn, Germany}
\email{research@altifina.com}

\author{Mykhaylo Tyomkyn}
\address{Department of Applied Mathematics, Faculty of Mathematics and Physics, Charles University}
\email{tyomkyn@kam.mff.cuni.cz}
\thanks{The second author is supported by GA\v{C}R grant 25-17377S}

\begin{document}
\begin{abstract}
We prove that for every integer $d\ge2$ and all sufficiently
large $m$, every digraph with $m$ edges and minimum outdegree
at least $d$ has a vertex bipartition for which each directed
cut is of size at least
\[\left\lceil\frac{d-1}{2(2d-1)}m+\frac{d^2}{2(2d-1)}\right\rceil.\]
The bound is sharp for infinitely many $m$.
This answers for large graphs a question of Scott
and proves a conjecture of Lee, Loh, and Sudakov. For $d=2$ we prove that $\lceil m/6\rceil$ is achievable for all $m$.
We also establish a stability result characterizing near-extremal graphs.

The proofs were obtained by LLMs, with human input. They have been revised by the authors and verified in Lean. The presentation is human-written.   

\end{abstract}
\maketitle

\section{Introduction}

A folklore fact in graph theory is that every finite undirected graph $G=(V,E)$ has a vertex partition $V=A\sqcupdot B$ such that $e(A,B)\geq |E|/2$. For example, one can assign each vertex independently with probability $1/2$ to either part, and take the expectation of the random variable $e(A,B)$. The same argument for directed graphs gives a partition in which the directed cut $e(A,B)$ is of size at least $|E|/4$.   

In a similar spirit, one can ask for many edges to cross in each
direction, but the example of an out-directed large star shows that further conditions on the graph must be imposed. In this regard, Scott~\cite{Scott} asked to find the largest constant $c=c(d)$ such that   
every digraph with $m$ edges and minimum outdegree at least $d$ admits a vertex bipartition in which both directed cuts are of size at least $cm$. 

Lee, Loh, and Sudakov~\cite{LLS} conjectured that, for every
fixed $d\ge2$, every digraph with $m$ edges and minimum
outdegree at least $d$ admits a vertex bipartition in which
both directed cuts have size at least $(c_d-o(1))m$ as
$m\to\infty$, where $c_d=(d-1)/(2(2d-1))$ (note that this is below $1/4$). 
To support their conjecture, they gave a construction showing that this coefficient cannot be
increased. Let us recall their construction, introducing a small but crucial improvement. 

A tournament is a digraph containing exactly one of $uv$ and
$vu$ for each pair of distinct vertices; it is regular if all outdegrees
are equal. For an integer $k\ge1$, take $k$ vertex-disjoint regular
tournaments of order $2d-1$ and one further vertex-disjoint regular
tournament of order $2d+1$. Choose a vertex $z$ in the latter tournament,
and add an edge to $z$ from each vertex of each smaller tournament. Moreover, for every edge $zu$ in the large tournament add the edge $uz$, for a total of $d$ extra edges (this is the aforementioned improvement -- we emphasize that it is not merely a trivial point overlooked in previous papers, since adding new edges also increases $m$, affecting the additive term adversely).
The resulting digraph $G$ has minimum outdegree at least $d$ and $m=kd(2d-1)+2d(d+1)$ edges. 
In every vertex bipartition $V(G)=A\sqcupdot B$, with $z\in A$, the cut $e(A,B)$ will be of size at most 
$$k\frac {d (d-1)}2+\frac{d(d+1)}2=c_dm+\frac{d(d+1)}{2(2d-1)}=\left\lceil c_dm+\frac{d^2}{2(2d-1)}\right\rceil.$$

The authors of~\cite{LLS} proved their conjecture for $d=2,3$,
and Liu and Yu~\cite{LY} subsequently proved it for $d=4$.
Related bounds under additional degree or structural assumptions were
obtained in~\cite{HW,HMYZ,HLW,WH}. Further related work includes results on bisections~\cite{HWY,LMZ, MWZ}. 

Our main goal in this paper is to prove an optimal bound matching the above construction. 
\begin{theorem}\label{thm:exact}
	Let $d\geq 2$ be a fixed integer. Then for every sufficiently large $m$ every digraph $G$ with $m$ edges and minimum outdegree at least $d$ admits a vertex partition $V(G)=A\sqcupdot B$ with
	\[
	\min\{e_G(A,B),e_G(B,A)\}\ge \left \lceil \frac{d-1}{2(2d-1)}m+\frac{d^2}{2(2d-1)}\right\rceil. 
	\]
\end{theorem}

Thus, for each fixed $d$, the construction and Theorem~\ref{thm:exact} determine the value in question exactly for all sufficiently large $m$ with $d\mid m$ and $\frac{m}{d}\equiv3\pmod{2d-1}$. 

Our proof follows a familiar pattern in extremal graph theory: we first prove an asymptotic result, convert it into a stability result, and use the latter to establish the exact result. Our asymptotic bound, in particular, proves the Lee--Loh--Sudakov conjecture for every $d$. 
\begin{theorem}\label{thm:asympt}
Let $G$ be a digraph with $m$ edges and minimum outdegree at
least $d\geq 2$. Then its vertex set has a partition $V(G)=A\sqcupdot B$ for
which the numbers $e_G(A,B)$ and $e_G(B,A)$ of edges from $A$ to $B$
and from $B$ to $A$, respectively, satisfy
\[
\begin{aligned}
e_G(A,B)&\ge \frac{d-1}{2(2d-1)}m-\frac52m^{6/7},\\
e_G(B,A)&\ge \frac{d-1}{2(2d-1)}m-\frac52m^{6/7}.
\end{aligned}
\]
\end{theorem}
\begin{theorem}\label{thm:stability}
	Let $d\ge2$ be a fixed integer, and let $(G_i)_{i\ge1}$ be a
	sequence of digraphs with minimum outdegree at least $d$.
	For each $i\ge1$, put $m_i=|E(G_i)|$.
	Suppose that $m_i\to\infty$ and
	\[
	\max_{V(G_i)=A\sqcupdot B}
	\min\{e_{G_i}(A,B),e_{G_i}(B,A)\}
	=c_dm_i+o(m_i).
	\]
	Then there exist spanning subdigraphs $H_i\subseteq G_i$ with $|E(G_i)\setminus E(H_i)|=o(m_i)$, and such that
each $H_i$ consists of vertex-disjoint regular tournaments of
		order $2d-1$, a single vertex $z_i$ outside these tournaments receiving
		an edge from every tournament vertex, and $o(m_i)$ isolated
		vertices. %All asymptotic statements refer to $i\to\infty$.
\end{theorem}

For $d=2$ we resolve Scott's question in full. The proof of the following theorem is given in Section~\ref{sec:exact} and can be read independently.

\begin{theorem}\label{thm:2out}
Let $G$ be a digraph with $m$ edges and minimum outdegree at least $2$. Then there is a partition $V(G)=A\sqcupdot B$ such that
\[
e_G(A,B)\ge\left\lceil\frac m6\right\rceil
\qquad\text{and}\qquad
e_G(B,A)\ge\left\lceil\frac m6\right\rceil.
\]
\end{theorem}

The rest of the paper is structured as follows. In Section~\ref{sec:probab} we establish an auxiliary star-partition lemma, and the main probabilistic lemma of the proof. In Sections~\ref{sec:asympt} and~\ref{sec:stability} we prove Theorems~\ref{thm:asympt} and~\ref{thm:stability}, respectively. In Section~\ref{sec:exact} we conclude the proof of Theorem~\ref{thm:exact}, and give a stand-alone proof of Theorem~\ref{thm:2out}. We end the paper with some concluding remarks in Section~\ref{sec:conclusion}. 
\subsection*{Proof overview}
Let us describe here the proof of Theorem~\ref{thm:asympt}, which introduces the main ideas. 

In a graph where all total degrees are small, say constant, we could simply partition the vertex set uniformly at random, and applying standard concentration bounds would yield both directed cuts of size close to $m/4$ with positive probability. Where this fails, e.g. in the above extremal construction, it is due to the presence of large degree vertices. Hence, the strategy is to treat the vertices of large and small total degree (we call their sets $X$ and $Y$ respectively) separately. With appropriate parameters we may delete all edges inside $X$ at a negligible loss, so that in the resulting graph $G'$ every edge touches $Y$. The goal now is to combine a deterministic partition of $X$ with a random one of $Y$.

The next tool was first used in the undirected setting by Erd\H{o}s, Gy\'arf\'as and Kohayakawa~\cite{EGK}. Note that, since linearity of expectation does not require independence, an extra gain can be achieved by vertex-partitioning the graph into induced bipartite graphs. Then in each of them assign all vertices in one part to $A$ and in the other one to $B$ -- and decide which way to do it with probability $1/2$ independently for each induced bipartite graph. In this way the edges between the bipartite graphs still contribute $1/2$ to the expectation, while the edges within the induced bipartite graphs contribute $1$, for a total gain over the uniform assignment. One non-trivial partition into bipartite graphs that always exists is a partition into induced stars and singletons, based on the depth-first exploration of the connected components of the graph. We apply this method in the directed setting similarly to~\cite{LLS}, but with an improvement. The resulting random partition of $Y$ gives a sufficient condition on the required deterministic partition of $X$. In particular, the remaining case to consider is when there are at least $3$ vertices in $X$ whose in-degrees significantly outweigh their out-degrees.

This and further structural information allow a new pass at partitioning $Y$ randomly, but this time with a bias (this is where our proof crucially deviates from that in~\cite{LLS}, where biased assignments were made independently for individual vertices). Given an induced star partition of the underlying graph of $G'[Y]$, we assign the vertices to parts $A$ and $B$ as follows. For each star independently we assign with probability $1-2p$ (for $p\leq 1/2$) all vertices to $A$, with probability $p$ its centre to $A$ and all leaves to $B$, and with probability $p$ its centre to $B$ and all leaves to $A$. Each singleton is assigned independently to $A$ with probability $1-p$ and to $B$ with probability $p$. Note that in this assignment every vertex belongs to $B$ with probability $p$, while each edge within a star contributes $p$, as opposed to  $p(1-p)$ under independence, to each directed-cut expectation. 

It turns out that the ensuing sufficient condition on the partition of $X$ can be realized deterministically. Among three vertices whose indegrees significantly outweigh
their outdegrees, choose a vertex $z$ of minimum indegree
in $G'$. A balancing argument gives two partitions of $X$
which keep $z$ in $A$ and reverse the placement of every
other vertex. For $p=(d-1)/(2d-1)$, the expected cut from
$A$ to $B$ is the smaller one for both partitions.
Averaging these two expectations and using the
minimum-outdegree assumption shows that at least one
partition satisfies the required deterministic condition.

\vspace{2mm}
\noindent
\textbf{Notation.} All digraphs in this paper are finite and have no loops and at most one edge
for each ordered pair of distinct vertices. Both edges $uv$ and $vu$ may be
present. The underlying graph of a digraph is the simple undirected graph
in which $u$ and $v$ are adjacent if at least one of these edges is present.
A component of a digraph $G$ is the subdigraph induced by the vertex set
of a connected component of its underlying graph.
A component is called \emph{simple} if it contains no pair of opposite edges.
In particular, isolated vertices are simple components.

For a digraph $G$, write $V(G)$ for its vertex set and $E(G)$ for its edge
set. If $S,T\subseteq V(G)$ are disjoint, let $e_G(S,T)$ be the number of
edges with tail in $S$ and head in $T$, and put $e_G(S)=|E(G[S])|$, where
$G[S]$ is the induced subdigraph on $S$. We write $e_G(v,T)$ for
$e_G(\{v\},T)$. The outdegree, indegree, and total degree of $v$ are
denoted by $d_G^+(v)$, $d_G^-(v)$, and
$d_G(v)=d_G^+(v)+d_G^-(v)$, respectively. Subscripts are omitted when the
digraph is clear.

\vspace{2mm}
\noindent
\textbf{Statement on AI use.} The proof of Theorem~\ref{thm:asympt} was first discovered by the Altifina research framework, which combines multiple AI models---in this work, OpenAI's ChatGPT 6 Astra and 6.1 Sol, as well as Anthropic's Claude Opus 5.5. The proof was revised and extended by the second author, aided by ChatGPT 6 Astra, leading to the main results of this paper, Theorems~\ref{thm:exact} and~\ref{thm:2out}. The proofs have been verified in Lean, and a certificate is available upon request. The presentation is human-written and proofread with ChatGPT 6 Astra.

\section{The probabilistic lemma}\label{sec:probab}

In this section we formulate a sufficient deterministic criterion for a suitable random partition of $G$ to yield the desired large cuts with positive probability. 

An induced star in a digraph means a vertex set whose induced underlying
graph is $K_{1,r}$ for some integer $r\ge 1$. A vertex adjacent to all
the other vertices is designated as its centre; the other $r$ vertices
are its leaves. For a star on two vertices, either endpoint may be designated as its
centre.

The following lemma, establishing a partition into induced stars with many edges inside the parts, dates back to Erd\H{o}s, Gy\'arf\'as and Kohayakawa~\cite{EGK}. For our purposes we need a small improvement to the bound when $H$ contains a pair of opposite edges.

\begin{lemma}\label{lem:connected-stars}
Let $H$ be a connected digraph on $c\ge1$ vertices. Then there is a
partition $\mathcal P$ of $V(H)$ into induced stars and at most one
singleton such that
\[
\sum_{P\in\mathcal P}e_H(P)\ge
\begin{cases}
(c-1)/2,&\text{if $H$ is simple},\\
c/2,&\text{if $H$ is not simple}.
\end{cases}
\]
\end{lemma}

\begin{proof}
If $H$ is simple, choose an arbitrary vertex $u$ as a root. Otherwise,
choose opposite edges $uv$ and $vu$ and take $u$ as the root.
Run depth-first search in the underlying graph of $H$ from $u$, obtaining
a rooted spanning tree $T$. Each edge of the underlying graph joins two
vertices one of which is an ancestor of the other in $T$. Indeed, if $x$
is discovered before a neighbour $y$, then $y$ is discovered before the
search finishes at $x$, and hence lies in the subtree rooted at $x$.

As long as the remaining tree has at least two vertices, choose a nonleaf
of maximum depth and remove it together with all its children as a part $P\in \mathcal P$.
All these children are leaves of the remaining tree. They are pairwise
nonadjacent in the underlying graph, since distinct children are not
ancestors of one another in $T$. Thus the removed part induces a star.
It is a pendant subtree of the remaining tree, so its removal leaves
either the empty graph or a tree containing $u$. If one vertex
remains at the end, it is $u$, and we take it as a singleton. Since each star has at least two vertices, the total number of underlying edges in the
stars is at least $(c-1)/2$. Each underlying edge corresponds to at least one edge of $H$. Thus the
stars contain at least $(c-1)/2$ edges of $H$, proving the assertion when
$H$ is simple.

Suppose that $H$ is not simple. If there is no singleton, the same count
gives at least $c/2$ edges. If there is a singleton, it must be $u$.
If, in addition, a star has at least three vertices, then, again we obtain at least $c/2$ underlying edges. 

Hence, it remains to consider the case in which $u$ is the singleton and all
stars have exactly two vertices. The underlying edges of these stars
form a matching covering $V(H)\setminus\{u\}$. In particular, $v$ belongs
to a star $\{v,w\}$. Replace this star by $\{u,v\}$, with $w$ becoming
the singleton instead of $u$. There are still $(c-1)/2$ stars on two
vertices, and the star $\{u,v\}$ contains two directed edges. Thus the stars
contain at least $(c-1)/2+1>c/2$ edges, as required.
\end{proof}

\noindent For a digraph $H$, let $\tau(H)$ be the number of its
simple components, and put
\[
L(H)=\frac{|V(H)|-\tau(H)}2.
\]
In particular, we have $0\le\tau(H)\le |V(H)|$ and $L(H)\ge0$; the number $L(H)$ need not be an integer.

\begin{corollary}\label{cor:stars}
Let $H$ be a digraph. Then there exists a partition
$\mathcal P$ of $V(H)$ into induced stars and singletons such that 
$\sum_{P\in\mathcal P}e_H(P)\ge L(H)$.
\end{corollary}
\begin{proof}
Apply Lemma~\ref{lem:connected-stars} separately to each component of $H$
and take the union of the resulting partitions. Writing $C$ for the
vertex set of a component, we obtain
\[
\begin{aligned}
	\sum_{P\in\mathcal P}e_H(P)
	&\ge\sum_{\substack{C:\,H[C]\text{ simple}}}\frac{|C|-1}{2}
	+\sum_{\substack{C:\,H[C]\text{ not simple}}}\frac{|C|}{2}\\
	&=\frac{|V(H)|-\tau(H)}2=L(H).
\end{aligned}
\]
\end{proof}

Let $G$ be a digraph with $V(G)=X\sqcupdot Y$, where $X$ is
independent and $Y\ne\emptyset$, and put $L=L(G[Y])$.
For an ordered partition $\mathcal X=(X_1,X_2)$ of $X$ and
$p\in[0,1/2]$, define
\begin{equation}\label{eq:ideals}
\begin{aligned}
I_{AB}(\mathcal X,p)
  &=p\,e(X_1,Y)+(1-p)e(Y,X_2)+p(1-p)e(Y)+p^2L,\\
I_{BA}(\mathcal X,p)
  &=(1-p)e(X_2,Y)+p\,e(Y,X_1)+p(1-p)e(Y)+p^2L.
\end{aligned}
\end{equation}
The first three terms are the expected cut sizes when $X_1$ is placed in
$A$, $X_2$ is placed in $B$, and each vertex of $Y$ is assigned
independently to $A$ with probability $1-p$ and to $B$ with probability $p$. The term $p^2L$ accounts for an improvement achieved by
assigning the vertices of each star in a suitable dependent way.

\begin{lemma}\label{lem:completion}
Let $G$ be a digraph with $m>0$ edges, and suppose $V(G)=X\sqcupdot Y$,
where $X$ is independent and $Y\ne\emptyset$. Let
$\mathcal X=(X_1,X_2)$ be an ordered partition of $X$, and let
$p\in[0,1/2]$ be a constant. Put $L=L(G[Y])$ and
$\Delta=\max_{w\in Y}d_G(w)$, and let $I_{AB}$ and $I_{BA}$ be
defined by \eqref{eq:ideals}. Then there is a bipartition $V(G)=A\sqcupdot B$ with $X_1\subseteq A$ and
$X_2\subseteq B$ such that
\[
\begin{aligned}
e_G(A,B)&\ge I_{AB}(\mathcal X,p)-\frac32(m^2\Delta)^{1/3},\\
e_G(B,A)&\ge I_{BA}(\mathcal X,p)-\frac32(m^2\Delta)^{1/3}.
\end{aligned}
\]
\end{lemma}

\begin{proof}
Since $X$ is independent,
every edge has an endpoint in $Y$. Put
\[
K=(m\Delta^2)^{1/3}.
\]
Then $K\ge\Delta$, since $m\ge\Delta$.

Consider the partition of $Y$ given by Corollary~\ref{cor:stars}.
Replace each star $P$ with $\sum_{w\in P}d_G(w)>K$ by its singleton
vertices. The stars being replaced are pairwise disjoint and
$\sum_{w\in Y}d_G(w)\le2m$, so their number is at most $2m/K$.
Every edge within a star is incident with its centre. Consequently, each
such star has at most $\Delta$ edges, and the number $M$ of
edges belonging to the surviving stars satisfies
\begin{equation}\label{eq:surviving}
M\ge L-\frac{2m\Delta}{K}.
\end{equation}
Write the resulting partition as $\{P_1,\ldots,P_r\}$ and put
$C_j=\sum_{w\in P_j}d_G(w)$. Every surviving star satisfies $C_j\le K$,
and the same holds for each singleton because $K\ge\Delta$.

Make the choices for distinct parts $P_j$ independently. For each star,
place all its vertices in $A$ with probability $1-2p$, its centre in $A$
and all its leaves in $B$ with probability $p$, and its centre in $B$
and all its leaves in $A$ with probability $p$. These assignments are
summarized in the table
\[
\begin{array}{c|c|c}
\text{centre}&\text{all leaves}&\text{probability}\\ \hline
A&A&1-2p\\
A&B&p\\
B&A&p
\end{array}
\]
For a star on two vertices, the distribution is unchanged
when its centre and leaf are exchanged.
For each singleton, place its vertex in $A$ with probability $1-p$ and
in $B$ with probability $p$. Place $X_1$ in $A$ and $X_2$ in $B$.
For every star, both the centre and each leaf
belong to $A$ with probability $(1-2p)+p=1-p$. Thus every vertex of $Y$
belongs to $A$ with probability $1-p$ and to $B$ with probability $p$.

Let $Z_1=e_G(A,B)$ and $Z_2=e_G(B,A)$ be the resulting random cut
sizes. We compute $\E (Z_1)$ by considering all possible locations of the
endpoints of an edge. An edge from $X_1$ to $Y$ contributes to $Z_1$
exactly when its head belongs to $B$, which has probability $p$.
An edge from $Y$ to $X_2$ contributes exactly when its tail belongs to
$A$, which has probability $1-p$. Edges from $X_2$ to $Y$ and from $Y$
to $X_1$ never contribute to $Z_1$, because their tail lies in $B$ or
their head lies in $A$, respectively. There are no edges with both
endpoints in $X$.

Now consider an edge with both endpoints in $Y$. If it belongs to a
surviving star, its endpoints are the centre and a leaf. When it is
directed from the centre to the leaf, it contributes to $Z_1$ in the
second row of the table; when it is directed from the leaf to the centre,
it contributes in the third row. In either case its contribution
probability is $p$. There are $M$ such edges. Each of the remaining
$e(Y)-M$ edges joins distinct parts $P_j$ of the partition of $Y$.
Its endpoints are assigned independently, so its tail belongs to $A$
and its head to $B$ with probability $p(1-p)$.

\noindent Linearity of expectation now gives
\[
\begin{aligned}
\E (Z_1)
&=p\,e(X_1,Y)+(1-p)e(Y,X_2)\\
&\quad +pM+p(1-p)\bigl(e(Y)-M\bigr)\\
&=p\,e(X_1,Y)+(1-p)e(Y,X_2)+p(1-p)e(Y)+p^2M\\
&=I_{AB}(\mathcal X,p)+p^2(M-L),
\end{aligned}
\]
where the last equality uses \eqref{eq:ideals}. Analogously,
\[
\E (Z_2)=I_{BA}(\mathcal X,p)+p^2(M-L).
\]
Using \eqref{eq:surviving} and $p\le1/2$, we obtain
\begin{equation}\label{eq:means}
\begin{aligned}
\E (Z_1)&\ge I_{AB}(\mathcal X,p)-\frac{2p^2m\Delta}{K}
        \ge I_{AB}(\mathcal X,p)-\frac{m\Delta}{2K},\\
\E (Z_2)&\ge I_{BA}(\mathcal X,p)-\frac{2p^2m\Delta}{K}
        \ge I_{BA}(\mathcal X,p)-\frac{m\Delta}{2K}.
\end{aligned}
\end{equation}
We next bound the variances. For each $j$, let $\mathcal A_j$ be the
set of edges having at least one endpoint in $P_j$. Then
$|\mathcal A_j|\le C_j$. Write $Z_i=\sum_{e\in E(G)}W_{i,e}$, where
$W_{i,e}$ is the $0/1$ random variable equal to $1$ if edge $e$
contributes to the respective cut and to $0$ otherwise.
The random variable $W_{i,e}$ depends only on the assignments of the
parts $P_j$ for which $e\in\mathcal A_j$.
Thus $W_{i,e}$ and $W_{i,f}$ are independent
unless $e$ and $f$ belong to a common $\mathcal A_j$. By
Cauchy--Schwarz and the variance bound for a $0/1$ random variable,
\[
|\Cov(W_{i,e},W_{i,f})|
\le\sqrt{\Var(W_{i,e})\Var(W_{i,f})}\le\frac14.
\]
There are at most $\sum_j|\mathcal A_j|^2$ ordered pairs $(e,f)$
belonging to a common $\mathcal A_j$. Hence, for each $i=1,2$,
\[
\begin{aligned}
\Var(Z_i)
&=\sum_{e,f\in E(G)}\Cov(W_{i,e},W_{i,f})\\
&\le\frac14\sum_{j=1}^r|\mathcal A_j|^2
\le\frac14\sum_{j=1}^r C_j^2
\le\frac K4\sum_{j=1}^r C_j
\le\frac{mK}{2}.
\end{aligned}
\]
It follows that
\[
\E\bigl((Z_1-\E (Z_1))^2+(Z_2-\E (Z_2))^2\bigr)\le mK.
\]
The pair $(Z_1,Z_2)$ takes only finitely many values, so some outcome
of positive probability satisfies
$(Z_1-\E (Z_1))^2+(Z_2-\E (Z_2))^2\le mK$. For that outcome,
$Z_i\ge\E (Z_i)-\sqrt{mK}$ for $i=1,2$. Using \eqref{eq:means}, we obtain
\[
\begin{aligned}
Z_1&\ge I_{AB}(\mathcal X,p)-\frac{m\Delta}{2K}-\sqrt{mK},\\
Z_2&\ge I_{BA}(\mathcal X,p)-\frac{m\Delta}{2K}-\sqrt{mK}.
\end{aligned}
\]
Finally,
\[
\frac{m\Delta}{2K}+\sqrt{mK}
=\frac12(m^2\Delta)^{1/3}+(m^2\Delta)^{1/3}
=\frac32(m^2\Delta)^{1/3}.
\]
This proves the lemma.
\end{proof}

\section{Finding the auxiliary partition}\label{sec:asympt}
\noindent
Now we aim to prove~Theorem~\ref{thm:asympt} by applying Lemma~\ref{lem:completion} with suitable parameters. 

Let $G$ be a digraph with $m$ edges
and minimum outdegree at least $d\geq 2$. Put
\[
X=\{v\in V(G):d_G(v)>m^{4/7}\} \ \text{ and } \ Y=V(G)\setminus X.
\]
Since $\sum_{v\in V(G)}d_G(v)=2m$, we have $|X|\le2m^{3/7}$, and so
\[
e_G(X)\le |X|(|X|-1)\le 4m^{6/7}.
\]
If $Y=\emptyset$, then $m=e(X)\le4m^{6/7}$. Since $c_d<1/4$, we then have
\[
c_dm-\frac52m^{6/7}\le m^{6/7}-\frac52m^{6/7}<0,
\]
which means every partition satisfies the required inequalities.

Assume $Y\ne\emptyset$, and let $G'$ be the spanning subdigraph
obtained by deleting all edges with both endpoints in $X$.
Then $X$ is independent in $G'$ and $m'=|E(G')|=m-e_G(X)$.
For every $w\in Y$, both its indegree and its outdegree are unchanged.
In particular, we have $d_{G'}^+(w)\ge d$ for every $w\in Y$, and
\[
\Delta:=\max_{w\in Y}d_{G'}(w)\le m^{4/7}.
\]
Put $y=|Y|$, $\tau=\tau(G'[Y])$ and $L=(y-\tau)/2$, and use
\eqref{eq:ideals} to define $I_{AB}$ and $I_{BA}$ for $G'$ and this
value of $L$. It suffices to find an ordered partition
$\mathcal X=(X_1,X_2)$ of $X$ and $p\in[0,1/2]$ such that
\begin{equation}\label{eq:target}
\min\{I_{AB}(\mathcal X,p),I_{BA}(\mathcal X,p)\}\ge c_dm'.
\end{equation}
Indeed, Lemma~\ref{lem:completion} then gives a partition
$V(G)=A\sqcupdot B$ for which each directed cut in $G'$ has size at least
\[
\begin{aligned}
c_dm'-\frac32(m'^2\Delta)^{1/3}
&\ge c_dm-c_de_G(X)-\frac32(m^2\cdot m^{4/7})^{1/3}\\
&\ge c_dm-\frac{1}{4}\cdot 4m^{6/7}-\frac32m^{6/7}
=c_dm-\frac52m^{6/7}.
\end{aligned}
\]
Restoring the deleted edges does not decrease either
cut size.

The graph $G'$ and the sets $X,Y$ remain fixed throughout this section.
All edge counts without a subscript refer to $G'$. Denote
\[ \alpha:=e(X,Y), \ \ \beta:=e(Y,X), \ \ \text{and} \ \ \gamma:=e(Y).
\]
Since $X$ is
independent in $G'$, we have
\begin{align}\label{eq:y-degree}
m'=\alpha+\beta+\gamma
\ge \alpha+\sum_{w\in Y}d_{G'}^+(w)
\ge \alpha+dy.
\end{align}
We need the following technical lemma.
\begin{lemma}\label{lem:degree}
For every $S\subseteq X$ we have
\begin{equation}\label{eq:degree}
	e(Y,X\setminus S)\ge(d-|S|)\tau-L.
\end{equation}
\end{lemma}

\begin{proof}
If $G'[Y]$ has no simple component, the right hand side is negative, and the conclusion holds automatically. Otherwise, for each simple component $C$ of $G'[Y]$, take a vertex $u_C\in C$ of minimum outdegree in $G'[Y]$. It satisfies
	\[
	d_{G'[Y]}^+(u_C)\le\frac{|C|-1}{2}.
	\]
	Let the set $U$ consist of these vertices, with no vertex chosen from the other
	components. We have  $|U|=\tau$ and
	\[
	\sum_{u\in U}d_{G'[Y]}^+(u)
	=\sum_{\substack{C:\,\text{ simple}}}d_{G'[Y]}^+(u_C)
	\le\sum_{\substack{C:\,\text{ simple}}}\frac{|C|-1}{2}
	\le L.
	\]
Every $u\in U$ sends at most
$|S|$ edges to $S$, so summing the outdegrees in $G'$ gives
\[
\begin{aligned}
d\tau\leq \sum_{u\in U} d^+_{G'}(u)
&=\sum_{u\in U}\bigl(d_{G'[Y]}^+(u)+e(u,S)+e(u,X\setminus S)\bigr)\\
&\le L+|S|\tau+e(Y,X\setminus S).
\end{aligned}
\]
The assertion follows by rearranging.
\end{proof}
Statements~(1) and~(2) of the following lemma are numerical forms of the balancing arguments
of Lee--Loh--Sudakov~\cite{LLS} and Liu--Yu~\cite{LY}.
Statement~(3) is an elementary consequence of the same greedy balancing
argument. 
\begin{lemma}\label{lem:signs}
Let $N\ge 0$, let $\ell$ be a nonnegative integer, and let
$a_1,\ldots,a_\ell\in[-N,N]$. Define
\[
\theta=\min_{(\sigma_1,\ldots,\sigma_\ell)\in\{-1,1\}^{\ell}}
\left|\sum_{i=1}^{\ell}\sigma_i a_i\right|.
\]
Suppose $\theta>N/2$. Put
$D=\{i:|a_i|\ge\theta\}$ and $g=\sum_{i\notin D}|a_i|$.
Then the following hold:
\begin{enumerate}
\item $|D|\notin\{0,2\}$;
\item $g\le N-\theta$;
\item for every $z\in D$, there exist signs $\sigma_i\in\{-1,1\}$ for
$i\ne z$ such that, writing $r=\sum_{i\ne z}\sigma_i a_i$, we have
\[
|a_z|+r\ge\theta\qquad\text{and}\qquad |a_z|-r\ge\theta.
\]
\end{enumerate}
\end{lemma}

\begin{proof}
For any finite list of numbers in $[-N,N]$, signs can be chosen greedily so that
the absolute value of their signed sum is at most $N$. Indeed,
process the numbers in any order, starting with sum $0$. Choose the
sign of the next term so that its contribution has the opposite sign
to the current sum, with either choice allowed when the current sum
or the term is $0$. The new absolute value is the absolute difference
between the old absolute value and the magnitude of the next term,
and so is again at most $N$.

The same rule shows that, if the current sum has absolute value less
than $\theta$ and every remaining term has magnitude less than
$\theta$, then all subsequent sums have absolute value less than
$\theta$. If $D=\emptyset$, applying this rule to the entire list
would give a full signed sum of absolute value less than $\theta$,
contrary to the definition of $\theta$. If $D=\{i,j\}$, first choose
signs for these two terms so that their sum has absolute value
\[
\bigl||a_i|-|a_j|\bigr|\le N-\theta<\theta.
\]
Applying the same rule to the remaining terms gives the same
contradiction. This proves (1).

For (2), first choose signs on $D$ so that their sum $A$ satisfies
$|A|\le N$, and then process all terms outside $D$ by the same rule.
Every intermediate sum, including $A$, has absolute value at least
$\theta$: otherwise all subsequent sums would have absolute value
less than $\theta$, again contradicting the definition of $\theta$.
Each remaining term has magnitude less than $\theta$, so adding it
cannot change the sign of the current sum and reduces
its absolute value by exactly the magnitude of that term. The final
sum therefore has absolute value $|A|-g$. Consequently, we have
\[
\theta\le |A|-g\le N-g,
\]
which gives $g\le N-\theta$.

Finally, fix $z\in D$, and apply the same sign rule to all indices other
than $z$, obtaining $r=\sum_{i\ne z}\sigma_i a_i$ with $|r|\le N$.
Both $|a_z|+r$ and $|a_z|-r$ are full signed sums of the original list.
Since $|a_z|\ge\theta$ and $\theta>N/2$, we have
\[
\min\{|a_z|+r,|a_z|-r\}
=|a_z|-|r|\ge\theta-N>-\theta.
\]
No full signed sum belongs to $(-\theta,\theta)$, by the definition of
$\theta$. Hence this smaller value is at least $\theta$, proving (3).
\end{proof}

Returning to the graphs, for $x\in X$, put $a_x=d_{G'}^-(x)-d_{G'}^+(x)$.
Since $X$ is independent, we have $|a_x|\le y$.
For an ordered partition $\mathcal X=(X_1,X_2)$ of $X$, let
$\sigma_x=1$ for $x\in X_1$ and $\sigma_x=-1$ for $x\in X_2$,
and put
\[
D_{\mathcal X}=\sum_{x\in X}\sigma_x a_x.
\]
Every choice of signs is realized by an ordered partition of $X$.
Define
\[
\theta=\min_{\mathcal X}|D_{\mathcal X}|,\qquad
Z^-=\{x\in X:a_x\ge\theta\} \ \text{ and } \
Z^+=\{x\in X:a_x\le-\theta\},
\]
where the minimum is over all ordered partitions of $X$.

\begin{lemma}\label{lem:fair}
	If $\theta\le m'/(2d-1)+L$, then there is an ordered partition
	$\mathcal X=(X_1,X_2)$ of $X$ such that
	\[
	\min\{I_{AB}(\mathcal X,1/2),I_{BA}(\mathcal X,1/2)\}
	\ge c_dm'.
	\]
\end{lemma}

\begin{proof}
	Let $\mathcal X=(X_1,X_2)$ be an ordered partition of $X$.
	By \eqref{eq:ideals}, we have
	\[
	\begin{aligned}
		I_{AB}(\mathcal X,1/2)+I_{BA}(\mathcal X,1/2)
		&=\frac{m'+L}{2},\\
		I_{BA}(\mathcal X,1/2)-I_{AB}(\mathcal X,1/2)
		&=\frac{D_{\mathcal X}}2.
	\end{aligned}
	\]
	Consequently, we have
	\[
	\min\{I_{AB}(\mathcal X,1/2),I_{BA}(\mathcal X,1/2)\}
	=\frac{m'+L-|D_{\mathcal X}|}{4},
	\]
	and the definition of $\theta$ gives
	\begin{equation}\label{eq:fair-optimum}
	\max_{\mathcal X}\min\{I_{AB}(\mathcal X,1/2),I_{BA}(\mathcal X,1/2)\}
	=\frac{m'+L-\theta}{4},
	\end{equation}
	where the maximum is taken over all ordered partitions of $X$.
	
    Choose a partition $\mathcal X$ attaining $|D_{\mathcal X}|=\theta$. This $\mathcal X$ satisfies
	\[
	\begin{aligned}
		\min\{I_{AB}(\mathcal X,1/2),I_{BA}(\mathcal X,1/2)\}
		&=\frac{m'+L-\theta}{4}\\
		&\ge\frac{m'-m'/(2d-1)}4
		=c_dm'.
	\end{aligned}
	\]
\end{proof}
\begin{lemma}\label{lem:few}
	If $\theta>m'/(2d-1)+L$, then $|Z^-|\ge3$.
\end{lemma}
\begin{proof}
	By \eqref{eq:y-degree} and $L\ge0$, we have
	\[
	\theta>\frac{m'}{2d-1}+L
	\ge\frac{dy}{2d-1}
	>\frac{y}{2}.
	\]
	Since $\theta>0$, the sets $Z^-$ and $Z^+$ are disjoint.
	Put
	\[
	g=\sum_{x\in X\setminus(Z^-\cup Z^+)}|a_x|.
	\]
	Apply Lemma~\ref{lem:signs} to the numbers $a_x$ with $N=y$.
	Its first two assertions give
	\[
	|Z^-|+|Z^+|\notin\{0,2\}
	\qquad\text{and}\qquad
	0\le g\le y-\theta.
	\]
	By the definition of $a_x$, we have
	\[
	\begin{aligned}
		\alpha-e(Y,X\setminus Z^-)
		&=\sum_{x\in Z^-}d_{G'}^+(x)
		+\sum_{x\in Z^+}(-a_x)\\
		&\quad+\sum_{x\in X\setminus(Z^-\cup Z^+)}(-a_x)\\
		&\ge |Z^+|\theta-g.
	\end{aligned}
	\]
	Using \eqref{eq:y-degree}, Lemma~\ref{lem:degree} with $S=Z^-$,
	and $y=2L+\tau$, we obtain
	\[
	\begin{aligned}
		m'+(2d-1)L
		&\ge dy+\alpha+(2d-1)L\\
		&\ge dy+\alpha-e(Y,X\setminus Z^-)+(d-|Z^-|)\tau+(2d-2)L\\
		&=(2d-1)y-(|Z^-|-1)\tau+\alpha-e(Y,X\setminus Z^-)\\
		&\ge (2d-1)y-(|Z^-|-1)\tau+|Z^+|\theta-g\\
		&\ge (2d-2)y-(|Z^-|-1)\tau+(|Z^+|+1)\theta.
	\end{aligned}
	\]
	Rearranging gives 
		\begin{equation}\label{eq:small-count}
		\begin{aligned}
			(|Z^-|-1)\tau
			&+m'+(2d-1)L-(2d-1)\theta\\
			&\ge (2d-2)(y-\theta)+|Z^+|\theta.
		\end{aligned}
	\end{equation}
	
	If $|Z^-|\le1$, then the left-hand side of \eqref{eq:small-count} is negative,
	whereas its right-hand side is nonnegative, a contradiction.
	If $|Z^-|=2$, the inequality
	$|Z^-|+|Z^+|\notin\{0,2\}$ gives $|Z^+|\ge1$.
	By \eqref{eq:small-count}, we then have
	\[
	\begin{aligned}
		\tau
		&>(2d-2)(y-\theta)+|Z^+|\theta\\
		&\ge (2d-2)(y-\theta)+\theta\\
		&=y+(2d-3)(y-\theta)\\
		&\ge y,
	\end{aligned}
	\]
	contrary to $\tau\le y$. Hence $|Z^-|\ge3$.
\end{proof}

\begin{lemma}\label{lem:paired}
	Suppose $\theta>y/2$ and $z\in Z^-$. Then there exist
	two ordered partitions $\mathcal X^{(1)}=(X_1^{(1)},X_2^{(1)})$ and $\mathcal X^{(2)}=(X_1^{(2)},X_2^{(2)})$ of $X$
	such that $z\in X_1^{(1)}\cap X_1^{(2)}$, every vertex of $X\setminus \{z\}$ lies in $X_1^{(1)}$ if and only if it lies in $X_2^{(2)}$, and 
	\[
	D_{\mathcal X^{(j)}}\ge\theta
	\]
	for $j=1,2$.
\end{lemma}

\begin{proof}
	Since $z\in Z^-$, we have $a_z\ge\theta>0$.
	By Lemma~\ref{lem:signs}(3), there exist signs
	$\sigma_x\in\{-1,1\}$, for each $x\in X\setminus\{z\}$,
	such that, writing
	\[
	r=\sum_{x\in X\setminus\{z\}}\sigma_xa_x,
	\]
	we have
	\[
	a_z+r\ge\theta
	\qquad\text{and}\qquad
	a_z-r\ge\theta.
	\]
	Place $z$ in $X_1^{(1)}$. For each $x\in X\setminus\{z\}$,
	place $x$ in $X_1^{(1)}$ if $\sigma_x=1$ and in $X_2^{(1)}$
	if $\sigma_x=-1$. Obtain $\mathcal X^{(2)}$ by keeping $z$
	in the first part and reversing the placement of every other vertex.
	Then we have
	\[
	D_{\mathcal X^{(1)}}=a_z+r\ge\theta
	\qquad\text{and}\qquad
	D_{\mathcal X^{(2)}}=a_z-r\ge\theta.
	\]
\end{proof}
Suppose a set $Z\subseteq X$ satisfies $|Z|=3$, and let $z\in Z$ be a vertex
whose indegree in $G'$ is minimal among the vertices of $Z$. Let $\mathcal X^{(1)}$
and $\mathcal X^{(2)}$ be two ordered partitions of $X$
which place $z$ in the first part and reverse the placement
of every other vertex.
Define, for $p\in[0,1/2]$, 
\[
\mu(p)=
\frac{I_{AB}(\mathcal X^{(1)},p)
	+I_{AB}(\mathcal X^{(2)},p)}2.
\]
Averaging \eqref{eq:ideals} over the two partitions gives
\begin{equation}\label{eq:mean-parameter}
		\mu(p)
		=\frac p2\left(\alpha+d_{G'}^+(z)\right)
		+\frac{1-p}{2}\left(\beta-d_{G'}^-(z)\right)
		+p(1-p)\gamma+p^2L.
\end{equation}
\begin{lemma}\label{lem:average}
	For every choice of $Z$, $z$, $\mathcal X^{(1)}$ and
	$\mathcal X^{(2)}$ satisfying the above conditions, and $p=(d-1)/(2d-1)$, we have
	\[
	\mu(p)\ge c_dm'.
	\]
\end{lemma}

\begin{proof}
Substituting $p=(d-1)/(2d-1)$ into \eqref{eq:mean-parameter}, and using notational shorthands $d^-=d_{G'}^-(z)$ and $d^+=d_{G'}^+(z)$, gives
\begin{equation*}
\begin{aligned}
	\mu(p)
	={}&\frac{d-1}{2(2d-1)}
	\bigl(\alpha+d^+\bigr)+\frac{d}{2(2d-1)}
	\bigl(\beta-d^-\bigr)+\frac{d(d-1)}{(2d-1)^2}\gamma
	+\frac{(d-1)^2}{(2d-1)^2}L.
\end{aligned}
\end{equation*}
Together with $m'=\alpha+\beta+\gamma$, this yields
\begin{equation}\label{eq:average-expanded}
\begin{aligned}
	2(2d-1)^2\bigl(\mu(p)-c_dm'\bigr)
	={}&(2d-1)(d-1)\bigl(\alpha+d^+\bigr)+d(2d-1)\bigl(\beta-d^-\bigr)\\
	&+2d(d-1)\gamma+2(d-1)^2L-(2d-1)(d-1)\bigl(\alpha+\beta+\gamma\bigr)\\
	={}&(2d-1)(d-1)d^+
	+(2d-1)\beta+(d-1)\gamma\\
	&-d(2d-1)d^-+2(d-1)^2L\\
	={}&(2d-1)(d-1)d^+
	+(d-1)\bigl(\beta+\gamma\bigr)\\
	&+d\bigl(\beta-(2d-1)d^-\bigr)
	+2(d-1)^2L.
\end{aligned}
\end{equation}
Since $X$ is independent and $|Z|=3$, we have
\begin{equation}\label{eq:indegree-decomposition}
\begin{aligned}
	\beta-(2d-1)d^-
	&=\sum_{x\in Z}d_{G'}^-(x)+e(Y,X\setminus Z)
	-(2d-1)d^-\\
	&=\sum_{x\in Z}\bigl(d_{G'}^-(x)-d^-\bigr)
	+e(Y,X\setminus Z)-2(d-2)d^-.
\end{aligned}
\end{equation}
Substituting \eqref{eq:indegree-decomposition} into
\eqref{eq:average-expanded} while adding and subtracting
$d(d-1)y$ gives
\begin{equation}\label{eq:average-degree-excess}
\begin{aligned}
	2(2d-1)^2\bigl(\mu(p)-c_dm'\bigr)
	={}&(2d-1)(d-1)d^+ + (d-1)\bigl(\beta+\gamma\bigr)+d\sum_{x\in Z}
	\bigl(d_{G'}^-(x)-d^-\bigr)\\
	&
	+d\,e(Y,X\setminus Z)-2d(d-2)d^-+2(d-1)^2L\\
	={}&(2d-1)(d-1)d^+ +(d-1)\bigl(\beta+\gamma-dy\bigr)+d\sum_{x\in Z}
	\bigl(d_{G'}^-(x)-d^-\bigr)\\
	&
	+d\,e(Y,X\setminus Z)+d(d-1)y-2d(d-2)d^-+2(d-1)^2L.
\end{aligned}
\end{equation}
Since $y=2L+\tau$, we have
\begin{equation}\label{eq:average-remainder}
\begin{aligned}
	&d(d-1)y-2d(d-2)d^-+2(d-1)^2L\\
	&\qquad=2d(d-2)\bigl(y-d^-\bigr)
	-d(d-3)y+2(d-1)^2L\\
	&\qquad=2d(d-2)\bigl(y-d^-\bigr)
	-d(d-3)\tau+(2d+2)L\\
	&\qquad=2d(d-2)\bigl(y-d^-\bigr)
	+d\bigl(L-(d-3)\tau\bigr)+(d+2)L.
\end{aligned}
\end{equation}
Substituting \eqref{eq:average-remainder} into
\eqref{eq:average-degree-excess}, we obtain
\begin{equation}\label{eq:average}
	\begin{aligned}
		2(2d-1)^2\bigl(\mu(p)-c_dm'\bigr)
		={}&(2d-1)(d-1)d^+ +(d-1)\bigl(\beta+\gamma-dy\bigr)\\
		&+d\sum_{x\in Z}
		\bigl(d_{G'}^-(x)-d^-\bigr)+d\bigl(e(Y,X\setminus Z)+L-(d-3)\tau\bigr)\\
		&+2d(d-2)\bigl(y-d^-\bigr)+(d+2)L.
	\end{aligned}
\end{equation}
	The second term is nonnegative by \eqref{eq:y-degree},
	the third by the choice of $z$, and the fourth by
	Lemma~\ref{lem:degree} with $S=Z$. The remaining terms
	are nonnegative because $d\ge2$, $d^-\le y$
	and $L\ge0$. This proves the lemma.
\end{proof}

\begin{proof}[Completion of the proof of Theorem~\ref{thm:asympt}]
	If $\theta\le m'/(2d-1)+L$, then Lemma~\ref{lem:fair}
	gives an ordered partition satisfying \eqref{eq:target}
	with $p=1/2$.
	
	Suppose now that $\theta>m'/(2d-1)+L$.
	By \eqref{eq:y-degree} and $L\ge0$, we have
	\[
	\theta>\frac{m'}{2d-1}+L
	\ge\frac{dy}{2d-1}
	>\frac{y}{2}.
	\]
	Lemma~\ref{lem:few} gives $|Z^-|\ge3$.
	Choose $Z\subseteq Z^-$ with $|Z|=3$, and let $z\in Z$
	have minimum indegree in $G'$ among the vertices of $Z$.
	Let $\mathcal X^{(1)}$ and $\mathcal X^{(2)}$ be the
	partitions supplied by Lemma~\ref{lem:paired}, and put
	$p=(d-1)/(2d-1)$. For each $j\in\{1,2\}$, the definitions
	in \eqref{eq:ideals} give
	\[
	\begin{aligned}
		I_{BA}(\mathcal X^{(j)},p)-I_{AB}(\mathcal X^{(j)},p)
		&=\frac{D_{\mathcal X^{(j)}}}{2}
		+\frac{\alpha-\beta}{2(2d-1)}\\
		&\ge\frac{\theta}{2}-\frac{m'}{2(2d-1)}\\
		&>\frac L2\ge0.
	\end{aligned}
	\]
	By Lemma~\ref{lem:average}, there is some $j\in\{1,2\}$
	such that $I_{AB}(\mathcal X^{(j)},p)\ge c_dm'$.
	The preceding inequality gives
	$I_{BA}(\mathcal X^{(j)},p)\ge c_dm'$ as well.
	Thus \eqref{eq:target} holds in both cases, and the reduction
	following \eqref{eq:target} proves Theorem~\ref{thm:asympt}.
\end{proof}

\section{Stability analysis}\label{sec:stability}

\begin{proof}[Proof of Theorem~\ref{thm:stability}]
	We follow the definitions and notation of Section~\ref{sec:asympt}.
	Fix $d$ and a sequence satisfying the assumptions of the theorem.
	We suppress the sequence index throughout the proof.
	All asymptotic statements refer to this index tending to infinity.
	
	Apply the construction from Section~\ref{sec:asympt} to each graph $G$.
	Thus
	\[
	X=\{v\in V(G):d_G(v)>m^{4/7}\}
	\qquad\text{and}\qquad
	Y=V(G)\setminus X,
	\]
	and $G'$ is obtained by deleting all edges with both endpoints
	in $X$. We have
	\[
	|X|\le2m^{3/7},\qquad
	m'=|E(G')|=m-O(m^{6/7}),
	\qquad\text{and}\qquad
	\max_{w\in Y}d_{G'}(w)\le m^{4/7}.
	\]
	All edge counts and degrees without a subscript refer to $G'$.
	
	Every vertex partition of $G'$ has a directed cut of size at
	most $c_dm'+o(m)$, with an error independent of the partition.
	Lemma~\ref{lem:completion} therefore gives
	\begin{equation}\label{eq:stability-ideals}
		\min\{I_{AB}(\mathcal X,p),I_{BA}(\mathcal X,p)\}
		\le c_dm'+o(m)
	\end{equation}
	for every ordered partition $\mathcal X$ of $X$ and every
	$p\in[0,1/2]$. The error is uniform in $\mathcal X$ and $p$,
	since the error in Lemma~\ref{lem:completion} is $O(m^{6/7})$.
	
	By~\eqref{eq:fair-optimum}, we have
	\begin{equation}\label{eq:stability-theta}
		\theta\ge\frac{m'}{2d-1}+L-o(m).
	\end{equation}
	The greedy signing argument in Lemma~\ref{lem:signs} gives
	$\theta\le y$. Together with \eqref{eq:y-degree}, this yields
	\[
	\frac{m'}{2d-1}-o(m)\le y\le\frac{m'}d.
	\]
	Consequently, we have $y=\Theta(m)$ and
	\[
	\theta-\frac{y}{2}
	\ge\frac{y}{2(2d-1)}-o(m)>0
	\]
	for all sufficiently large indices.
	
	We shall use the nonnegative quantities
	\[
	\varepsilon=\beta+\gamma-dy
	\qquad\text{and}\qquad
	\eta(S)=e(Y,X\setminus S)+L-(d-|S|)\tau
	\quad(S\subseteq X).
	\]
	Their nonnegativity follows from \eqref{eq:y-degree} and
	Lemma~\ref{lem:degree}. Put
	\[
	g=\sum_{x\in X\setminus(Z^-\cup Z^+)}|a_x|.
	\]
	By Lemma~\ref{lem:signs}, we have $g\le y-\theta$, and hence
	\[
	\alpha-e(Y,X\setminus Z^-)
	\ge |Z^+|\theta-g
	\ge |Z^+|\theta-(y-\theta).
	\]
	Using $y=2L+\tau$, we obtain
	\begin{equation}\label{eq:stability-budget}
		\begin{aligned}
			&(|Z^-|-1)\tau+m'+(2d-1)L-(2d-1)\theta\\
			&\quad=(2d-1)(y-\theta)
			+\alpha-e(Y,X\setminus Z^-)
			+\varepsilon+\eta(Z^-)\\
			&\quad\ge(2d-2)(y-\theta)+|Z^+|\theta
			+\varepsilon+\eta(Z^-);
		\end{aligned}
	\end{equation}
this is inequality \eqref{eq:small-count} with the two nonnegative
degree excesses included.
	
	\medskip
	\noindent\emph{Claim 1. For all sufficiently large indices,
		we have $|Z^-|\le2$.}
	
	Suppose otherwise, and pass to a subsequence on which
	$|Z^-|\ge3$. Choose $Z\subseteq Z^-$ with $|Z|=3$, and let
	$z\in Z$ have minimum indegree among the vertices of $Z$.
	Let
	$\mathcal X^{(j)}=(X_1^{(j)},X_2^{(j)})$, for $j\in\{1,2\}$,
	be the partitions supplied by Lemma~\ref{lem:paired}.
	Put $p_0=(d-1)/(2d-1)$, and consider their average $\mu(p)$.
	
	For each $j\in\{1,2\}$, Lemma~\ref{lem:paired} and
	\eqref{eq:stability-theta} give
	\[
	\begin{aligned}
		I_{BA}(\mathcal X^{(j)},p_0)-I_{AB}(\mathcal X^{(j)},p_0)
		&\ge\frac{\theta}{2}
		+\frac{\alpha-\beta}{2(2d-1)}\\
		&\ge\frac{\alpha}{2d-1}
		+\frac{\gamma}{2(2d-1)}+\frac L2-o(m)\\
		&\ge-o(m).
	\end{aligned}
	\]
	It follows from \eqref{eq:stability-ideals} that
	$I_{AB}(\mathcal X^{(j)},p_0)\le c_dm'+o(m)$.
	On the other hand, Lemma~\ref{lem:average} gives
	$\mu(p_0)\ge c_dm'$. Thus
	\[
	\mu(p_0)=c_dm'+o(m)
	\qquad\text{and}\qquad
	I_{AB}(\mathcal X^{(j)},p_0)=c_dm'+o(m)
	\quad(j=1,2).
	\]
	
	Suppose first that $d=2$. Retaining the fourth and last
	terms on the right-hand side of \eqref{eq:average}, we obtain
	\[
	\begin{aligned}
		18\bigl(\mu(p_0)-c_2m'\bigr)
		&\ge2\bigl(e(Y,X\setminus Z)+L+\tau\bigr)+4L\\
		&\ge2 y.
	\end{aligned}
	\]
	This contradicts $y=\Theta(m)$.
	
	Suppose now that $d\ge3$. Every term on the right-hand side
	of \eqref{eq:average} is nonnegative, and their sum is $o(m)$.
	Consequently, we have
	\[
	L=o(m),\qquad d^+(z)=o(m),\qquad
	d^-(z)=y+o(m),
	\qquad\text{and}\qquad
	\varepsilon=o(m).
	\]
	By the choice of $z$ and Lemma~\ref{lem:degree}, we have
	\[
	\begin{aligned}
		\beta
		&\ge3d^-(z)+e(Y,X\setminus Z)\\
		&\ge3d^-(z)+(d-3)\tau-L\\
		&=dy+o(m).
	\end{aligned}
	\]
	Together with $\varepsilon=o(m)$, this gives
	\[
	\beta=dy+o(m),\qquad
	\gamma=o(m),
	\qquad\text{and}\qquad
	m'=\alpha+dy+o(m).
	\]
	
	We now use
	\[
	p_1=\frac{d-1}{2d}.
	\]
	For each $j\in\{1,2\}$, expanding
	$I_{AB}(\mathcal X^{(j)},p_0)=c_dm'+o(m)$ gives
	\[
	(d-1)e(X_1^{(j)},Y)+d\,e(Y,X_2^{(j)})
	=\frac{d-1}{2}m'+o(m).
	\]
	Substituting this into the definitions of $I_{AB}$ and $I_{BA}$,
	and using $\gamma,L=o(m)$, yields
	\[
	\begin{aligned}
		d\bigl(I_{BA}(\mathcal X^{(j)},p_1)
		-I_{AB}(\mathcal X^{(j)},p_1)\bigr)
		&=\alpha-e(X_1^{(j)},Y)+o(m)\\
		&=e(X_2^{(j)},Y)+o(m)\\
		&\ge-o(m).
	\end{aligned}
	\]
	
	By \eqref{eq:stability-theta} and $\theta\le y$, we have
	$m'\le(2d-1)y+o(m)$. Hence
	\[
	\beta-c_dm'
	\ge\frac{d+1}{2}y+o(m)>0.
	\]
	Applying asymptotic extremality to the partition
	$V(G')=X\sqcupdot Y$, we obtain
	$\alpha\le c_dm'+o(m)$. Therefore
	\[
	\begin{aligned}
		(d-1)y-\alpha
		&=\frac{d-1}{d}m'
		-\frac{2d-1}{d}\alpha+o(m)\\
		&\ge\frac{d-1}{2d}m'+o(m).
	\end{aligned}
	\]
	Equation~\eqref{eq:mean-parameter} gives, uniformly for
	$p\in[0,1/2]$,
	\[
	\mu(p)=\frac p2\alpha
	+\frac{1-p}{2}(d-1)y+o(m).
	\]
	Since $\mu(p_0)=c_dm'+o(m)$, we obtain
	\[
	\begin{aligned}
		\mu(p_1)-c_dm'
		&=\frac{d-1}{4d(2d-1)}
		\bigl((d-1)y-\alpha\bigr)+o(m)\\
		&\ge\frac{(d-1)^2}{8d^2(2d-1)}m'+o(m).
	\end{aligned}
	\]
	Choose $j$ for which
	$I_{AB}(\mathcal X^{(j)},p_1)\ge\mu(p_1)$.
	For this partition,
	$I_{BA}(\mathcal X^{(j)},p_1)\ge
	I_{AB}(\mathcal X^{(j)},p_1)-o(m)$.
	Thus both quantities exceed $c_dm'$ by a positive multiple
	of $m$, contradicting \eqref{eq:stability-ideals}.
	This proves Claim 1.
	
	\medskip
	\noindent\emph{Claim 2. For all sufficiently large indices,
		there is a vertex $z$ such that $Z^-=\{z\}$ and $Z^+=\emptyset$.
		Moreover, we have
		$\theta=y+o(m)$ and $\varepsilon=\eta(\{z\})=o(m)$.}
	
	By Claim 1 and Lemma~\ref{lem:few}, we have
	$\theta\le m'/(2d-1)+L$. Together with
	\eqref{eq:stability-theta}, this gives
	\begin{equation}\label{eq:stability-tight}
		m'+(2d-1)L-(2d-1)\theta=o(m).
	\end{equation}
	Lemma~\ref{lem:signs} gives
	$|Z^-|+|Z^+|\notin\{0,2\}$.
	
	If $Z^-=\emptyset$, then $|Z^+|\ge1$.
	The left-hand side of \eqref{eq:stability-budget} is at most
	$o(m)$, whereas its right-hand side is at least
	$\theta=\Omega(m)$, a contradiction.
	
	Suppose that $|Z^-|=2$ along an infinite subsequence.
	Then $|Z^+|\ge1$. Substituting $\tau=y-2L$ into
	\eqref{eq:stability-budget} and using
	\eqref{eq:stability-tight}, we obtain
	\[
	2L+(2d-3)(y-\theta)+(|Z^+|-1)\theta
	+\varepsilon+\eta(Z^-)\le o(m).
	\]
	All terms on the left-hand side are nonnegative. Hence
	\[
	L=o(m),\qquad y-\theta=o(m),\qquad
	\varepsilon=\eta(Z^-)=o(m),
	\]
	and $|Z^+|=1$ for all sufficiently large indices on this
	subsequence. Each vertex of $Z^-$ has indegree $y+o(m)$.
	By the definition of $\eta(Z^-)$, we have
	\[
	\begin{aligned}
		\beta
		&=\sum_{x\in Z^-}d^-(x)+e(Y,X\setminus Z^-)\\
		&=2y+(d-2)\tau-L+o(m)\\
		&=dy+o(m).
	\end{aligned}
	\]
	Since $\varepsilon=o(m)$, we also have $\gamma=o(m)$.
	The two vertices of $Z^-$ have $a_x=y+o(m)$, the vertex
	of $Z^+$ has $a_x=-y+o(m)$, and the sum of the magnitudes
	of the remaining $a_x$ is at most $y-\theta=o(m)$.
	Consequently, we have
	\[
	\beta-\alpha=\sum_{x\in X}a_x=y+o(m).
	\]
	Thus,
	\[
	\alpha=(d-1)y+o(m)
	\qquad\text{and}\qquad
	m'=(2d-1)y+o(m).
	\]
	Both directed cuts of $X\sqcupdot Y$ have size at least
	$(d-1)y+o(m)$, whereas
	$c_dm'=(d-1)y/2+o(m)$.
	This contradicts asymptotic extremality.
	
	Hence $Z^-=\{z\}$ for all sufficiently large indices.
	Now \eqref{eq:stability-budget} and
	\eqref{eq:stability-tight} give
	\[
	(2d-2)(y-\theta)+|Z^+|\theta
	+\varepsilon+\eta(\{z\})\le o(m).
	\]
	The conclusions of Claim 2 follow.
	
	Since $a_z\ge\theta$ and $d^-(z)\le y$, we have
	\[
	d^+(z)=o(m)
	\qquad\text{and}\qquad
	d^-(z)=y+o(m).
	\]
	Lemma~\ref{lem:signs} also gives
	$\sum_{x\in X\setminus\{z\}}|a_x|=o(m)$.
	Summing the degree differences over $X\setminus\{z\}$,
	we obtain
	\begin{equation}\label{eq:stability-balance}
		\alpha=e(Y,X\setminus\{z\})+o(m).
	\end{equation}
	
	\medskip
	\noindent\emph{Claim 3. All but $o(m)$ vertices of $Y$
		belong to singleton components or regular tournament
		components of order $2d-1$ in $G'[Y]$. Each of these vertices
		sends an edge to $z$ and has outdegree exactly $d$ in $G'$.}
	
	For each component of $G'[Y]$, write $C$ for its vertex set
	and put
	\[
	\rho(C)=
	\begin{cases}
		\displaystyle e(C,X\setminus\{z\})
		+\frac{|C|-(2d-1)}2,
		& G'[C]\text{ is simple},\\[6pt]
		\displaystyle e(C,X\setminus\{z\})+\frac{|C|}{2},
		& G'[C]\text{ is not simple}.
	\end{cases}
	\]
	By Claim 2, we have
	\begin{equation}\label{eq:stability-components}
		\sum_C\rho(C)=\eta(\{z\})=o(m).
	\end{equation}
	For a simple component, the minimum-outdegree assumption gives
	\[
	\begin{aligned}
		d|C|
		&\le e(C)+e(C,\{z\})+e(C,X\setminus\{z\})\\
		&\le\frac{|C|(|C|-1)}2+|C|
		+e(C,X\setminus\{z\}).
	\end{aligned}
	\]
	Consequently, we have
	\[
	\rho(C)\ge
	\begin{cases}
		\displaystyle\frac{(|C|-1)(2d-1-|C|)}2,
		& |C|\le2d-1,\\[6pt]
		\displaystyle\frac{|C|-(2d-1)}2,
		& |C|>2d-1.
	\end{cases}
	\]
	Every $\rho(C)$ is therefore nonnegative, and it can vanish
	only for a simple component of order $1$ or $2d-1$.
	
	If $|C|=1$ and $\rho(C)=0$, then
	$e(C,X\setminus\{z\})=d-1$.
	The minimum-outdegree assumption forces the vertex to send
	an edge to $z$, so its outdegree is exactly $d$.
	If $|C|=2d-1$ and $\rho(C)=0$, then
	$e(C,X\setminus\{z\})=0$, and equality holds throughout
	the displayed degree estimate.
	Thus $G'[C]$ is a tournament, every vertex sends an edge
	to $z$, and every vertex has outdegree exactly $d$ in $G'$.
	Its internal outdegree is therefore $d-1$, so the tournament
	is regular.
	
	Every positive $\rho(C)$ is at least $|C|/(4d)$.
	Indeed, for a simple component with $2\le|C|\le2d-2$,
	the displayed bound is at least $|C|/4$.
	For $|C|\ge2d$, it is at least $|C|/(4d)$.
	For $|C|\in\{1,2d-1\}$, a positive value is an integer
	and hence is at least $1$.
	A nonsimple component contributes at least $|C|/2$.
	
	Let $W$ be the union of the components with $\rho(C)>0$.
	Equation~\eqref{eq:stability-components} gives $|W|=o(m)$,
	proving Claim 3.
	Moreover, since every vertex of $Y$ has outdegree at least $d$,
	we have
	\[
	\sum_{w\in W}d^+(w)
	\le d|W|+\sum_{w\in Y}\bigl(d^+(w)-d\bigr)
	=d|W|+\varepsilon=o(m).
	\]
	
	Let $S$ be the set of singleton components with $\rho(C)=0$,
	identifying each such component with its vertex, and let $k$
	be the number of tournament components with $\rho(C)=0$.
	Then
	\[
	y=(2d-1)k+|S|+|W|.
	\]
	
	\medskip
	\noindent\emph{Claim 4. We have $|S|=o(m)$.}
	
	Each vertex of $S$ sends exactly $d-1$ edges to
	$X\setminus\{z\}$, whereas the tournament vertices
	send no edges to this set.
	Since the total outdegree of $W$ is $o(m)$, we have
	\[
	e(Y,X\setminus\{z\})=(d-1)|S|+o(m).
	\]
	By \eqref{eq:stability-balance}, we also have
	$\alpha=(d-1)|S|+o(m)$.
	Using $\varepsilon=o(m)$ and $|W|=o(m)$, we obtain
	\[
	\begin{aligned}
		m'
		&=\alpha+dy+o(m)\\
		&=d(2d-1)k+(2d-1)|S|+o(m).
	\end{aligned}
	\]
	
	Put $z$ and every vertex of $S\cup W$ in $A$, and put
	$X\setminus\{z\}$ in $B$.
	Independently for distinct tournaments, choose a
	uniformly random subset of $d$ vertices and place it in $A$,
	placing the remaining $d-1$ vertices in $B$.
	
	For every choice, each tournament contributes
	$d(d-1)/2$ edges in each direction.
	Indeed, regularity makes its two directed cut sizes equal,
	and their sum is $d(d-1)$.
	All edges from $S$ to $X\setminus\{z\}$ go from $A$ to $B$,
	so every outcome satisfies
	\[
	e(A,B)\ge k\frac{d(d-1)}2+(d-1)|S|.
	\]
	An edge from $X\setminus\{z\}$ to a tournament
	contributes to $e(B,A)$ with probability $d/(2d-1)$,
	and an edge from $X\setminus\{z\}$ to $S\cup W$
	contributes with probability $1$.
	Since $d^+(z)=o(m)$, we have
	\[
	\begin{aligned}
		\E (e(B,A))
		&\ge k\frac{d(d-1)}2
		+\frac{d}{2d-1}e(X\setminus\{z\},Y)\\
		&=k\frac{d(d-1)}2
		+\frac{d(d-1)}{2d-1}|S|+o(m).
	\end{aligned}
	\]
	Choose an outcome for which $e(B,A)$ is at least its
	expectation. The bound on $e(A,B)$ holds for this outcome
	as well. Since
	\[
	c_dm'=k\frac{d(d-1)}2+\frac{d-1}{2}|S|+o(m),
	\]
	we obtain
	\[
	\min\{e(A,B),e(B,A)\}
	\ge c_dm'+\frac{d-1}{2(2d-1)}|S|+o(m).
	\]
	Asymptotic extremality implies $|S|=o(m)$, proving Claim 4.
	
	It follows that
	\[
	\alpha=o(m)
	\qquad\text{and}\qquad
	e(Y,X\setminus\{z\})=o(m).
	\]
	The total outdegree of $S\cup W$ is also $o(m)$.
	Delete from $G'$ every edge whose tail lies in $X$,
	every edge from $Y$ to $X\setminus\{z\}$, and every edge
	whose tail lies in $S\cup W$.
	The number of deleted edges is $o(m)$.
	The resulting spanning subdigraph consists precisely of
	the regular tournaments of order $2d-1$,
	all edges from their vertices to $z$, and isolated vertices.
	The number of isolated vertices is
	\[
	|X|-1+|S|+|W|=o(m).
	\]
	Including the $O(m^{6/7})$ edges initially deleted inside $X$,
	the total number of edges deleted from $G$ is $o(m)$.
	
	Restoring the sequence index gives the required subdigraphs
	$H_i$ for all sufficiently large $i$.
	For the finitely many remaining indices, take an edgeless
	spanning subdigraph.
\end{proof}

\section{The exact results}\label{sec:exact}

\begin{lemma}\label{lem:rooted-cut}
	Let $G$ be a digraph with $m$ edges and minimum outdegree at least
	$d\ge2$, and let $z\in V(G)$. Then there is a partition
	$V(G)=A\sqcupdot B$ with $z\in A$ such that
	\[
	e_G(A,B)\ge c_dm+\frac{d\,d_G^+(z)}{2(2d-1)}.
	\]
\end{lemma}

\begin{proof}
	Put $H=G-z$. Every vertex of $H$ has outdegree at least $d-1$.
	For each simple component $C$ of $H$, we have
	\[
	(d-1)|V(C)|\le |E(C)|
	\le \frac{|V(C)|(|V(C)|-1)}2.
	\]
	Thus every simple component has at least $2d-1$ vertices, and
	\[
	\tau(H)\le\frac{|V(H)|}{2d-1}.
	\]
	By Corollary~\ref{cor:stars}, there is a partition $\mathcal P$
	of $V(H)$ into induced stars and singletons such that
	\[
	\sum_{P\in\mathcal P}e_H(P)
	\ge L(H)
	\ge\frac{d-1}{2d-1}|V(H)|.
	\]
	Put $z$ in $A$. Independently for distinct members of $\mathcal P$,
	assign their vertices as follows. For each star, put its centre in
	$A$ and all its leaves in $B$ with probability $1/2$, and reverse
	these assignments with probability $1/2$. For each singleton,
	put its vertex in either part with probability $1/2$.
	
	Each edge within a star contributes to $e_G(A,B)$ with probability
	$1/2$. Each edge of $H$ joining distinct members of $\mathcal P$
	contributes with probability $1/4$, and each edge leaving $z$
	contributes with probability $1/2$. Hence
	\[
	\begin{aligned}
		\E (e_G(A,B))
		&=\frac{|E(H)|}{4}
		+\frac14\sum_{P\in\mathcal P}e_H(P)
		+\frac{d_G^+(z)}2\\
		&\ge\frac{|E(H)|}{4}
		+\frac{d-1}{4(2d-1)}|V(H)|
		+\frac{d_G^+(z)}2.
	\end{aligned}
	\]
	The minimum-outdegree assumption gives
	\[
	|E(H)|+d_G^-(z)
	=\sum_{v\in V(H)}d_G^+(v)
	\ge d|V(H)|.
	\]
	Using $m=|E(H)|+d_G^+(z)+d_G^-(z)$ and
	$d_G^-(z)\le |V(H)|$, we obtain
	\[
	\begin{aligned}
		\E (e_G(A,B))-c_dm-\frac{d\,d_G^+(z)}{2(2d-1)}
		&\ge
		\frac{|E(H)|+(d-1)|V(H)|-2(d-1)d_G^-(z)}
		{4(2d-1)}\\
		&\ge\frac{|V(H)|-d_G^-(z)}4\\
		&\ge0.
	\end{aligned}
	\]
	Choose an outcome for which $e_G(A,B)$ is at least its expectation.
	This proves the lemma.
\end{proof}

\begin{proof}[Proof of Theorem~\ref{thm:exact}]
	Fix $d\ge2$, and suppose that the assertion is false.
	Then there is a sequence $(G_i)_{i\ge1}$ of digraphs with
	minimum outdegree at least $d$ such that, writing
	$m_i=|E(G_i)|$, we have $m_i\to\infty$ and
	\begin{equation}\label{eq:exact-counterexample}
		\max_{V(G_i)=A\sqcupdot B}
		\min\{e_{G_i}(A,B),e_{G_i}(B,A)\}
		<c_dm_i+\frac{d^2}{2(2d-1)}.
	\end{equation}
	Theorem~\ref{thm:asympt} bounds the same maximum below by
	$c_dm_i-(5/2)m_i^{6/7}$. Consequently, the sequence satisfies
	the assumptions of Theorem~\ref{thm:stability}.
	
	Apply Theorem~\ref{thm:stability}, and suppress the sequence
	index throughout the rest of the proof. All asymptotic
	statements refer to this sequence, with $d$ fixed.
	There is a spanning subdigraph $H\subseteq G$ consisting of
	vertex-disjoint regular tournaments of order $2d-1$, a vertex
	$z$ receiving an edge from every tournament vertex, and isolated
	vertices, such that
	\[
	|E(G)\setminus E(H)|=o(m).
	\]
	Put $D=E(G)\setminus E(H)$. Mark each tournament that contains
	an endpoint of an edge in $D$. At most $2|D|$ tournaments are
	marked. Let $R$ be the subdigraph of $G$ induced by $z$, all
	isolated vertices of $H$, and all vertices of the marked
	tournaments.
	
	Every edge of $D$ has both endpoints in $R$. The edges of
	$H[R]$ consist of the edges within the marked tournaments
	and their edges to $z$. Each such tournament contributes
	$d(2d-1)$ edges. We therefore have
	\[
	|E(R)|
	\le \bigl(1+2d(2d-1)\bigr)|D|
	=o(m).
	\]
	Every vertex of $R$ retains all its outgoing edges from $G$.
	Indeed, an edge from $V(R)$ to an unmarked tournament would
	be absent from $H$, and hence would belong to $D$, contrary
	to that tournament being unmarked. Thus $R$ has minimum
	outdegree at least $d$.
	
	Let $k$ be the number of unmarked tournaments. No edge of $D$
	has an endpoint in any of them. Hence their edges in $G$ are
	precisely their tournament edges and their edges to $z$, and
	\[
	m=|E(R)|+kd(2d-1).
	\]
	It follows that
	\[
	k\sim\frac{m}{d(2d-1)}
	\qquad\text{and}\qquad
	|E(R)|=o(k).
	\]
	Apply Lemma~\ref{lem:rooted-cut} to $R$ and $z$. Since
	$d_R^+(z)\ge d$, we obtain a partition
	$V(R)=A_R\sqcupdot B_R$ with $z\in A_R$ such that
	\[
	e_R(A_R,B_R)
	\ge c_d|E(R)|+\frac{d^2}{2(2d-1)}.
	\]
	Extend this partition to $V(G)=A\sqcupdot B$ by placing
	$d$ vertices of each unmarked tournament in $A$ and its
	remaining $d-1$ vertices in $B$.
	
	For each unmarked tournament, regularity implies that its
	two directed cut sizes are equal. Their sum is $d(d-1)$,
	so the tournament contributes $d(d-1)/2$ edges in each
	direction. Its edges to $z$ contribute a further $d-1$
	edges from $B$ to $A$. Consequently, we have
	\[
	\begin{aligned}
		e_G(A,B)
		&\ge c_d|E(R)|+\frac{d^2}{2(2d-1)}
		+k\frac{d(d-1)}2\\
		&=c_dm+\frac{d^2}{2(2d-1)}.
	\end{aligned}
	\]
	In the other direction, we have
	\[
	\begin{aligned}
		e_G(B,A)
		&\ge k\frac{d(d-1)}2+k(d-1)\\
		&=c_dm-c_d|E(R)|+k(d-1).
	\end{aligned}
	\]
	Since $|E(R)|=o(k)$ and $k\to\infty$, we have
	\[
	k(d-1)-c_d|E(R)|
	\ge\frac{d^2}{2(2d-1)}
	\]
	for all sufficiently large indices. Thus both directed cuts
	have size at least
	$c_dm+d^2/(2(2d-1))$, contradicting
	\eqref{eq:exact-counterexample}.
\end{proof}
\begin{proof}[Proof of Theorem~\ref{thm:2out}]
	Since $d^+(v)\le |V(G)|-1$ and every other vertex has
	outdegree at least $2$, we have
	\[
	m\ge d^+(v)+2(|V(G)|-1)\ge3d^+(v)
	\]
	for every $v\in V(G)$.
	
Suppose first that some vertex $z$ satisfies
$d^-(z)\ge m/3+1$. Then, by Lemma~\ref{lem:rooted-cut}, there is a partition
$V(G)=A_0\sqcupdot B_0$ with $z\in A_0$ and
$e(A_0,B_0)\ge m/6$.
Choose $\lceil m/6\rceil$ edges from $A_0$ to $B_0$,
let $A$ consist of $z$ and their tails, and put
$B=V(G)\setminus A$.
Every chosen edge still goes from $A$ to $B$, since its
head lies in $B_0$ and $A\subseteq A_0$.
Since the collection contains fewer than $m/6+1$ edges,
we have $|A\setminus\{z\}|<m/6+1$.
Consequently, we have
\[
\begin{aligned}
	e(A,B)&\ge\frac m6,\\
	e(B,A)&\ge e(B,\{z\})
	\ge d^-(z)-|A\setminus\{z\}|
	>\frac m6.
\end{aligned}
\]

We may therefore assume that $d^-(v)<m/3+1$
for every $v\in V(G)$. To handle this case, we use a simplified form of the pairing argument
of Ma, Wu, and Zeng~\cite{MWZ}.

Put $b(v)=d^+(v)-d^-(v)$ for each $v\in V(G)$.
Since $d^+(v)\ge2$, we have
\[
b(v)\le d^+(v)\le\frac m3
\qquad\text{and}\qquad
-b(v)<\frac m3+1-2<\frac m3.
\]
Thus $|b(v)|\le m/3$ for every vertex.
	
List the vertices with $b(v)>0$ in non-increasing order of
$b(v)$ and pair consecutive vertices.
Do the same for the vertices with $b(v)<0$, in non-increasing
order of $|b(v)|$.
Independently for each pair, put one vertex in $A$ and
the other in $B$, choosing either assignment with
probability $1/2$.
Assign every unpaired vertex, including those with
$b(v)=0$, independently to either part with probability
$1/2$. Each edge whose endpoints were paired up crosses the partition
with probability $1$, and every other edge crosses with
probability $1/2$. Hence we have
\[
\E\bigl(e(A,B)+e(B,A)\bigr)\ge\frac m2.
\]	
Put
\[
P=\sum_{\substack{v\in A\\b(v)>0}}b(v)
-\sum_{\substack{v\in B\\b(v)>0}}b(v)
\qquad\text{and}\qquad
Q=\sum_{\substack{v\in A\\b(v)<0}}b(v)
-\sum_{\substack{v\in B\\b(v)<0}}b(v).
\]
The random variables $P$ and $Q$ are independent and each has a distribution symmetric about zero.
	
For either (nonempty) list, let $a_1,a_2,\ldots$ be the
values $|b(v)|$ in non-increasing order, and append
a zero if the length of the list is odd.
	The absolute value of the corresponding random variable
	is at most
	\[
	(a_1-a_2)+(a_3-a_4)+\cdots\le a_1\le\frac m3.
	\]
	Thus $|P|,|Q|\le m/3$.
	By independence and symmetry, we have
	\[
	\begin{aligned}
		\E(|P+Q|)
		&=\frac12\E\bigl(|P+Q|+|P-Q|\bigr)\\
		&=\E(\max\{|P|,|Q|\})\\
		&\le\frac m3.
	\end{aligned}
	\]
	Since $\sum_{v\in V(G)}(d^+(v)-d^-(v))=0$, we have
	\[
	e(A,B)-e(B,A)
	=\sum_{v\in A}\bigl(d^+(v)-d^-(v)\bigr)
	=\frac{P+Q}{2}.
	\]
Consequently, we have
\[
\begin{aligned}
	\min\{e(A,B),e(B,A)\}
	&=\frac{e(A,B)+e(B,A)}2
	-\frac{|e(A,B)-e(B,A)|}{2}\\
	&=\frac{e(A,B)+e(B,A)}2-\frac{|P+Q|}{4}.
\end{aligned}
\]
Taking expectations, we obtain
\[
\begin{aligned}
	\E\bigl(\min\{e(A,B),e(B,A)\}\bigr)
	&=\frac12\E\bigl(e(A,B)+e(B,A)\bigr)
	-\frac14\E\bigl(|P+Q|\bigr)\\
	&\ge\frac12\cdot\frac m2-\frac14\cdot\frac m3\\
	&=\frac m6.
\end{aligned}
\]
Some outcome therefore has both directed cuts of size
at least $m/6$.
Their integrality gives the required bound
$\lceil m/6\rceil$.
\end{proof}
\section{Concluding remarks}\label{sec:conclusion}

Let us mention a number of related results that we obtained, but decided not to include in this paper for brevity and presentation reasons. 

Restating Theorem~\ref{thm:stability} in parameter form and using the resulting quantitative bounds in the proof of Theorem~\ref{thm:exact} gives a sufficient lower bound of $m(d)=(10d)^{50}$ in Theorem~\ref{thm:exact}. Furthermore, in this range it is possible, following the proof and using a standard derandomization technique, to devise a deterministic polynomial algorithm for constructing a partition claimed in Theorem~\ref{thm:exact}. The proof of Theorem~\ref{thm:2out} can also be converted into a polynomial algorithm.

Stating Theorem~\ref{thm:exact} only for sufficiently large $m$ appears necessary, as it fails by $1$ for $d=2$ and $m=9$ (note that it does not contradict Theorem~\ref{thm:2out}). The counterexample consists of a bidirected triangle and another vertex sending edges to all three triangle vertices. On the other hand, it seems plausible that $c_dm$ can be achieved for all $m$. We make it explicit in the following conjecture that constitutes the remaining open part of Scott's question (answered for large $m$ and for $d=2$ in Theorems~\ref{thm:exact} and~\ref{thm:2out}, respectively). It was already asked by Hou, Ma, Yu and Zhang~\cite{HMYZ} as a question, with emphasis on $d=2,3$. 

\begin{conjecture}
Let $d\geq 2$ be an integer, and let $G$ be a digraph with $m$ edges and minimum outdegree at least $d$. Then there is a partition $V(G)=A\sqcupdot B$ such that
\[
e_G(A,B)\ge\left\lceil\frac {d-1}{2(2d-1)}m\right\rceil
\qquad\text{and}\qquad
e_G(B,A)\ge\left\lceil\frac {d-1}{2(2d-1)}m\right\rceil.
\]
\end{conjecture}	 

Furthermore, we have obtained a stability result for extremal configurations in Theorem~\ref{thm:exact}. Any sufficiently large graph $G=(V,E)$ attaining the bound (rounded up) has a vertex subset $W\subseteq V(G)$ such that $G[V\setminus W]$ is a disjoint union of regular $(2d-1)$-vertex tournaments. A special vertex $z\in W$ receives an edge from every vertex in $V\setminus W$ and there are no further edges between $W$ and $V\setminus W$. The `core' digraph $C=G[W]$ itself has minimum out-degree at least $d$, and  
$$\max_{\substack{V(C)=A\sqcupdot B\\ z\in A}} e_C(A,B)=\left\lceil
\frac{(d-1)|E(C)|+d^2}{2(2d-1)}
\right\rceil.
$$
Note that the last condition is not merely a restatement of extremality, but concerns a one-sided cut containing a fixed vertex constrained to lie in $A$. The set $W$ can be chosen with fewer than $32d$ vertices. It must have at least $\max\{d+1,2d-5\}$ vertices, and we have constructed admissible cores of orders $2d$ and $14d-8$.

For every fixed $d\ge2$, the rounded bound in
Theorem~\ref{thm:exact} is sharp for infinitely many $m$, and can
be improved by $1$ for infinitely many other $m$. We do not know
whether, for some fixed $d$, it can be improved by $2$ for
infinitely many $m$. On the other hand, for every fixed $d$ and
all sufficiently large $m$, there is a construction whose
value exceeds the bound by at most
$\lceil(d+3)/4\rceil$.


\begin{thebibliography}{11}
\bibitem{EGK}
Paul Erd\H{o}s, Andr\'as Gy\'arf\'as, and Yoshiharu Kohayakawa,
\emph{The size of the largest bipartite subgraphs},
Discrete Math. \textbf{177} (1997), 267--271.

\bibitem{HLW}
Jianfeng Hou, Zhe Li, and Shufei Wu,
\emph{On bipartitions of directed graphs with small semidegree},
European J. Combin. \textbf{84} (2020), 103039.

\bibitem{HMYZ}
Jianfeng Hou, Huawen Ma, Xingxing Yu, and Xia Zhang,
\emph{A bound on judicious bipartitions of directed graphs},
Sci. China Math. \textbf{63} (2020), no.~2, 297--308.

\bibitem{HW}
Jianfeng Hou and Shufei Wu,
\emph{Bipartitions of oriented graphs},
J. Combin. Theory Ser. B \textbf{132} (2018), 107--133.

\bibitem{HWY}
Jianfeng Hou, Shufei Wu, and Guiying Yan,
\emph{On bisections of directed graphs},
European J. Combin. \textbf{63} (2017), 44--58.

\bibitem{LLS}
Choongbum Lee, Po-Shen Loh, and Benny Sudakov,
\emph{Judicious partitions of directed graphs},
Random Structures Algorithms \textbf{48} (2016), no.~1, 147--170;
\href{https://arxiv.org/abs/1305.2986}{arXiv:1305.2986}.

\bibitem{LMZ}
Guanwu Liu, Jie Ma, and Chunlei Zu,
\emph{Optimal bisections of directed graphs},
Random Structures Algorithms \textbf{64} (2024), no.~1, 138--153;
\href{https://arxiv.org/abs/2302.04050}{arXiv:2302.04050}.

\bibitem{LY}
Guanwu Liu and Xingxing Yu,
\emph{Partitioning digraphs with outdegree at least 4},
J. Graph Theory \textbf{98} (2021), no.~4, 604--622;
\href{https://arxiv.org/abs/2006.01116}{arXiv:2006.01116}.

\bibitem{MWZ}
Zhaoyang Ma, Shufei Wu, and Qinghou Zeng,
\emph{Sharp bisection bounds for digraphs},
preprint (2026),
\href{https://arxiv.org/abs/2608.24095}{arXiv:2608.24095}.

\bibitem{Scott}
Alex Scott,
\emph{Judicious partitions and related problems},
Surveys in Combinatorics 2005 (Bridget S. Webb, ed.),
London Math. Soc. Lecture Note Ser., vol.~327,
Cambridge University Press, Cambridge, 2005, pp.~95--117.

\bibitem{WH}
Shufei Wu and Jianfeng Hou,
\emph{On judicious bipartitions of directed graphs},
Discrete Math. \textbf{345} (2022), no.~10, 112988.
\end{thebibliography}
\end{document}